\documentclass[12pt]{amsart}
\usepackage{amscd}
\usepackage{verbatim}
\usepackage[russian,english]{babel}
\usepackage{amssymb, amsmath, amsthm, amscd,ifthen}
\usepackage[dvips]{graphics}
\usepackage[cp866]{inputenc}
\usepackage{graphicx}
\usepackage{epsfig}
\usepackage{hyperref}

\usepackage{amsmath,amssymb,amscd,}
\usepackage[all]{xy}

\usepackage{color}

\usepackage[dvips]{graphics}
\usepackage[cp866]{inputenc}
\usepackage{graphicx}
\usepackage{epsfig}

\theoremstyle{plain}
\newtheorem{theorem}{Theorem}
\newtheorem{lemma}{Lemma}

\newtheorem{proposition}{Proposition}
\newtheorem{remark}{Remark}

\theoremstyle{definition}
\newtheorem{definition}{Definition}
\newtheorem{example}{Example}

\theoremstyle{plain}
\newtoks\thehProclaim
\newtheorem*{Proclaim}{\the\thehProclaim}

\theoremstyle{definition}
\newtoks{\thehRemark}
\newtheorem*{Remark}{\the\thehRemark}

\renewcommand{\leq}{\leqslant}
\renewcommand{\geq}{\geqslant}

\begin{document}

	\title{Quasisections of circle bundles and Euler class}

	\author{ Gaiane Panina, Timur Shamazov,  Maksim Turevskii }
	
	\address{ G. Panina: St. Petersburg Department of Steklov Mathematical Institute; St. Petersburg State University;  gaiane-panina@rambler.ru; M.Turevskii:St. Petersburg State University; turmax20052005@gmail.com; Timur Shamazov:St. Petersburg State University;  }

	\subjclass[2000]{}
	

	\begin{abstract} Let $ E \xrightarrow[\text{}]{\pi} B$ be an oriented circle bundle
		over an oriented closed surface $B$. A \textit{quasisection}  is a smooth surface ${Q}$ (either closed or bordered)
		mapped by a generic smooth mapping  $q$  to $E$ such that $\pi\circ q({Q})=B$.
		
		In the paper we derive a \textit{local formula for the Euler number}, that is, we show that Euler number (Euler class) of the  bundle equals the sum of weights of (some of) singularities  of a quasisection.
		We also prove the uniqueness of  such a formula.
		
		The local formula is a close relative of  M. Kazarian's formula which relates the Euler number and Morse bifurcations
		of a generic function defined on the total space $E$. 
		
	\end{abstract}

	\maketitle

	\section{Introduction}\label{SecIntro}

	Let $ E \xrightarrow[\text{}]{\pi} B$   be a \textit{circle bundle},  that is, an oriented locally trivial  fiber bundle whose fibers are circles $S^1$. 
	Assume that the base $B$ is a closed $2$-dimensional oriented surface. Such bundles are classified by their Euler classes, or just \textit{Euler numbers}.
	
\medskip

We start with \textit{closed quasisections}:
	\begin{definition}
		A (closed) quasisection of $ E \xrightarrow[\text{}]{\pi} B$ is a closed smooth manifold  $Q$ and a smooth map
		$q:Q\rightarrow E$
		such that $\pi\circ q (Q)=B$. 

		The pair $(Q,q)$ will be denoted by $\mathcal{Q}$. We also abbreviate  the composition $\pi\circ q:Q\rightarrow B$ as $\pi:{\mathcal{Q}}\rightarrow B$.
		
	\end{definition}
	Throughout the paper, the maps $q$ and  $\pi\circ q $ are assumed to be \textit{generic}, or \textit{stable},  in the sense of Whitney's singularity theory \cite{AVGZ}. In particular, this means that the singularities of $q$ can be
	lines of self-intersection, isolated triple points, and Whitney umbrellas only. We also assume that away from  Whitney umbrellas, the singularities
	of  $\pi\circ q $  are pleats and folds. There are some more natural genericity conditions; we discuss them in due time.
	
	\medskip
\textbf{Main results}\begin{itemize}
                     
                       \item Theorem \ref{ThmMain} ({The  local formula for the Euler number  via singularities of a closed quasisection}) asserts that the Euler number of a bundle equals the weighted number of some special singularities of the  quasisection, called   \textit{singular vertices}. These are: ($I$) intersections of (projections of) two folds of $\mathcal{Q}$, 
	($II$) pleats of $\mathcal{Q}$, and
	($III$) intersections of a fold and a regular sheet of $\mathcal{Q}$. 
  \item Section \ref{SecExamples}  gives  a zoo of closed quasisections.
                       \item Theorem \ref{ThmUniqueness} (\textit{the uniqueness theorem}) asserts that  no other assignement of  weights  gives a valid  formula  in Theorem \ref{ThmMain}.

\item \textit{Bordered quasisections} and their singularities are studied  in Sec. \ref{SecBordered}. We  prove   a local formula forthe Euler number via singularities of a bordered quasisection, Theorem \ref{ThmMain2}.
                     \end{itemize}

	\medskip

\textbf{Motivations}

	Let us now mention relatives of our  results and some motivations.
	
	In \cite{Kaz}, the following construction is used.
	Given a generic smooth function $f:E\rightarrow \mathbb{R}$,  the restriction of $f$ to a fiber has at least two critical points.
	The set of critical points taken over all the fibers is a close analogue of a quasisection, but generic singularities in that framework look different. So what we prove is a very close relative of
	Theorem 1.4, (4) from \cite{Kaz}.

	We would also like to mention the
	\textit{local combinatorial formulae} for the Euler class, \cite{Igusa, Mnev3},
	that retrieve the Euler class of the bundle  from local data of a triangulated circle bundle.
	It was proven in \cite{Sonya} that this formula is not unique, so for triangulated bundles the analogue  of Theorem \ref{ThmUniqueness}
does not hold.

	Stable maps between closed surfaces have been studied since long. A stable map has folds and pleats only; the fold lines may have self-intersections in the image. There exist some relationships between the Euler characteristics of the surfaces,
	number of pleats, number  (and types) of crosses of the fold in the image, etc. \cite{Eliashberg, Quine, Yamamoto}. 
	For instance, it is known \cite{Thom, Whitney} that a generic map from $\mathbb{R}P^2$ to $\mathbb{R}^2$ necessarily has an odd number of pleats.
	
	Theorem \ref{ThmMain} can be also considered as a relationship between  weighted number of singularities and the Euler class of the fiber bundle. In our formula we have ''old'' singularities, (crossings of folds and pleats) counted with  some weights, and also ''new'' singularities (crossings of a fold and a regular sheet) which are not visible via the map $\pi\circ q$.  

\bigskip
	
	\textbf{A reminder on Euler class}
	
	We remind the reader that a circle bundle is uniquely defined by its Euler class  $\mathcal{E}(E\rightarrow B)\in H^2(B,\mathbb{Z})$. 
	Fixing an orientation of $B$ and identifying $H^2(B,\mathbb{Z})$ with $\mathbb{Z}$, we speak of the \textit{Euler number} $\mathcal{E}$ of a bundle $E\rightarrow B$.

	We use the following way to compute the Euler number:
	 Assume
	that a partial section $s: B\setminus\{x_i\}\rightarrow E$ is defined everywhere except for a finite set of points $x_1,...,x_m\in B$.
	For each of $x_i$  take a neighborhood $U_i$ bounded by a small circle $C_{x_i}$. The circle inherits an orientation from $B$.  
	Next, choose a trivialization of the bundle in the neighborhood $U_i$.
	The  restriction of the section $s$
	defines  a map   $$s_{\mid_{C_{x_i}}}:C_{x_i}\rightarrow S^1.$$  The source and the target are two oriented circles, 
	so the degree of the map is well-defined. Set $$ind_s(x_i):= deg~ s_{\mid_{C_{x_i}}}.$$ 
	The index $ind_s(x_i)$ depends on the section $s$, but the sum of the indices over all $x_i$ does not:
	
	\begin{proposition} \cite{Milnor}, \cite{FomFu}.
		In the above notation,$$\mathcal{E}(E\rightarrow B)=\sum_i ind_s(x_i).$$
	\end{proposition}

	Now we can explain \textbf{the leading idea  of our construction in short} (it traces back to M. Kazarian's multisections \cite{Kaz}).

	Assume we have a random section $s$  defined outside some fixed finite set  $\{x_i\}\subset B$. 
	Then one can take the expectation:

	$$\mathcal{E}(E\rightarrow B)=\mathbb{E}\Big(\sum_i ind_s(x_i)\Big)= \sum_i \mathbb{E}(ind_s(x_i)).$$
	
	We will show that a quasisection  yields a random section with a finite probability space. In other words,  a quasisection gives a finite collection of partial sections, and taking $\mathbb{E}$ amounts to averaging.

	\section{Singular points and their portraits}
	
	\subsection{Singularities}
	Let $\mathcal{Q}=(Q,q)$ be a quasisection of a circle bundle. The  \textit{singular curve} of $\mathcal{Q}$  is the image under projection $\pi$   of the line of self-intersection and  folds of  $\mathcal{Q}$.

	We assume in the paper that singularities of the singular curve  are stable, that is, persist under small perturbations of $q$.
	For instance, (the projections of) two pleats never coincide, there are no triple crossing points of (projections of) folds, etc.

	\medskip
	\textbf{Remark.}  Stability means that intersection of projections of two folds is always transversal. But the projection of a line of self-intersection may be 
	stably tangent to the projection of a fold.  This happens whenever a regular sheet of $\mathcal{Q}$ crosses a fold  in $E$.
	
	\medskip

	A \textit{singular vertex} is one of the following:
	\begin{enumerate}
		\item A self-intersection  of the singular curve.
		\item The image under projection $\pi$ of  a Whitney umbrella.
		\item The image under projection $\pi$ of a pleat.
	\end{enumerate}
	
	Singular vertices cut  the singular curve into \textit{singular segments} (or \textit{singular circles}, if a component of the singular curve is closed and has no singular vertices).

	A typical landscape of singular curves and singular vertices of a generic quasisection is depicted in Fig. \ref{fig:landscape}.
	
	\begin{figure}[h]
		\includegraphics[width=0.6\linewidth]{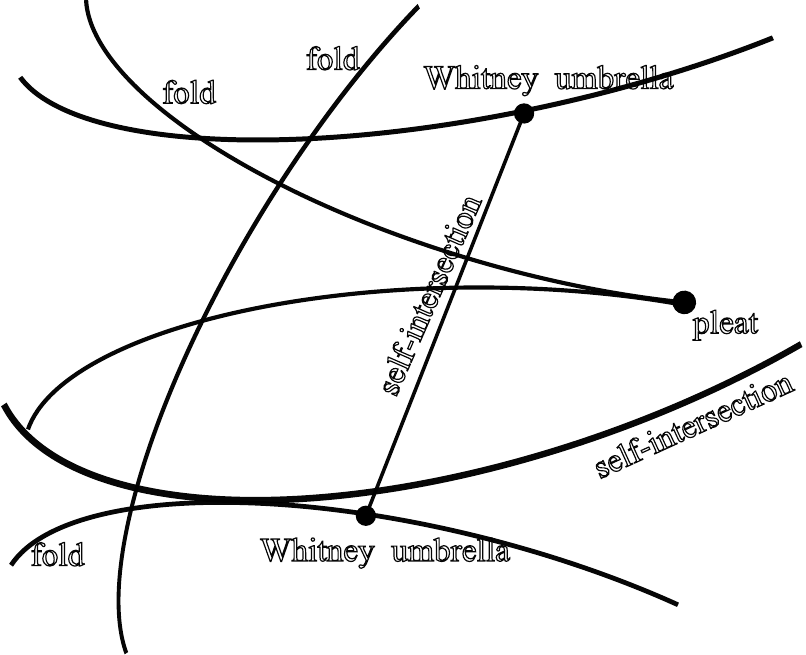}
		\caption {Singular curve and singular vertices, a fragment of a typical landscape.}
		\label{fig:landscape}
	\end{figure}

	\subsection{Portraits of points from the base}\label{SecPortraits}
	
	Let $x\in B$ be an arbitrary point.  We may assume that its neighborhood and the preimage of the neighborhood are equipped with flat Riemanian metrics
	such that the projection $\pi$ locally equals  the orthogonal projection, and the fibers have one and the same length. Let $C_x$ be a small circle embracing $x$. 
	Then $\pi^{-1}(C_x)$ is a flat torus.
	The circle $C_x$ is oriented according to the orientation of $B$; the orientation is depicted  counterclockwise. We depict 
	$\pi^{-1}(C_x)$  as an  annulus assuming
	that its inner and the outer circles are glued. Besides, we assume that the orientations of the fibers  point inwards,   from the outer circle to the inner circle,
	see Fig. \ref{fig:Ann}, left. In the annulus, we mark $\pi^{-1}(C_x)\cap\mathcal{ Q}$,
	 which is a non-emply union of some  curves.
	
	This picture  is called \textit{the portrait of the point} $x\in B$.  We say that  \textit{ two portraits are equal}
if they differ on orientation preserving fiberwise diffeomorphism.

	\bigskip
	
	\textbf{ Examples} \begin{enumerate}
                      \item If $x$ is a regular point, its portrait  (Fig. \ref{fig:Ann}, left) is a disjoint union of a non-zero number of \textit{simple circles}. 
	Each simple circle comes from a regular sheet of $\mathcal{Q}$.
                      \item Assume $x$ belongs to the interior of a singular edge which comes from a fold. Then its portrait looks as is depicted in Fig. \ref{fig:Ann}, right.
	It consists of  a (non-zero) number of simple circles, and one \textit{fold curve}.
	
                    \end{enumerate}

	\begin{figure}[h]
		\includegraphics[width=0.7\linewidth]{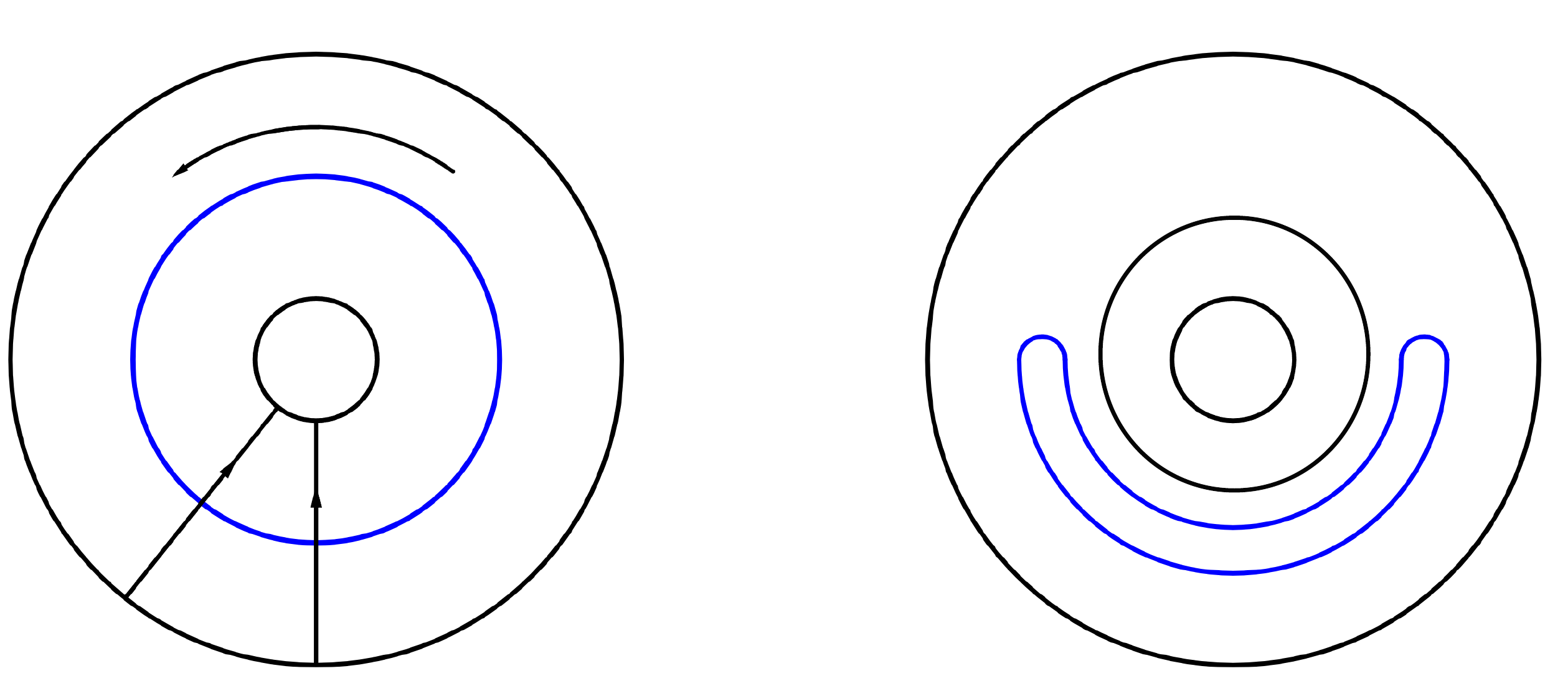}
		\caption { Left: The portrait of a regular point. Right:  A  portrait of a point $x$ lying on a fold. }
		\label{fig:Ann}
	\end{figure}
	\medskip

\newpage
	It is easy to see:
	\begin{lemma}\label{LemmaPortr}
		\begin{enumerate}
			\item If $x$ ranges within  the interior of one and the same singular edge (so $x$ is not a singular vertex), its portrait does not change.
			\item  Portrait of an intersection point of (projections of) two folds is  as  in Fig. \ref{fig:nonsymm}, left, up to adding/removing  some of simple circles.

 Such a singular vertex   is called  \textbf{type $I$}.
			
			\item  Portrait of a (projection of a)  pleat is (up to line  symmetry) as  in Fig. \ref{fig:nonsymm}, center. The number of simple
			circles may vary.

Such a singular vertex  $x$  is called \textbf{ type $II$, right}   (respectively,  \textbf{left}),
			if its portrait equals to (respectively, is line symmetric to) Fig.  \ref{fig:nonsymm},  center.

			\item  Portrait of the projection of an intersection of a fold and  a regular sheet is (up to line  symmetry) as in Fig.  \ref{fig:nonsymm}, right.
			
			Such a singular vertex  $x$ is called  \textbf{ type $III$, right} (respectively, \textbf{left}) if its  portrait equals to (respectively, is line symmetric to)   Fig. \ref{fig:nonsymm}, right, also up to  some number  of simple circles.  
			
		\end{enumerate}
		
	\end{lemma}

	Singular vertices of types $I$, $II$, and $III$ are called \textit{essential  vertices}.

	\begin{figure}[h]
		\includegraphics[width=1\linewidth]{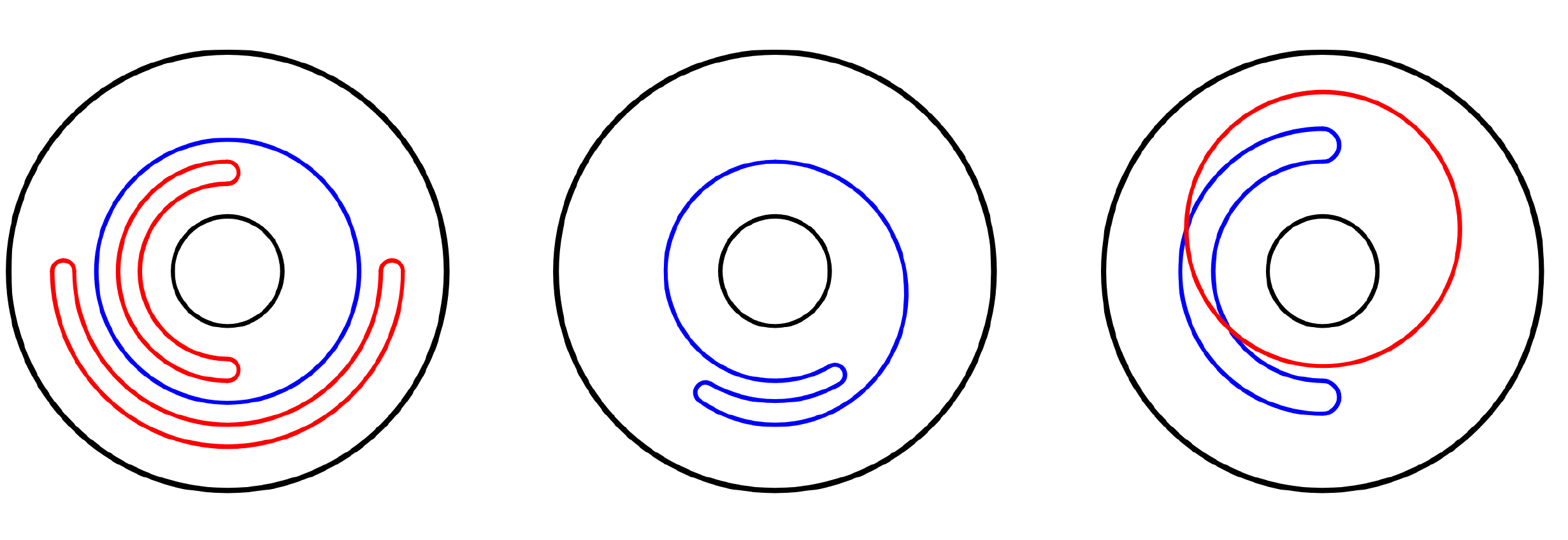}
		\caption {The list of all non-symm portraits, types $I$, $II$  right, and $III$  right. }
		\label{fig:nonsymm}
	\end{figure}
	
	Now consider  the projection of a Whitney umbrella.  Let us remind that a Whitney umbrella is defined 
	(in suitable coordinates, and up to diffeomorphisms of the target and the source  spaces) as $x(u,v)=uv, \ y(u,v)=u,\ z(u,v)=v^2$. The point $W=(0,0)$  is  called the \textit{endpoint} of the umbrella, 
	and the line of self-intersection  emanating from $W$ is called the\textit{ handle}.

	We are interested in the portrait of  (projection of) the endpoint $W$ under assumption that the  projection goes  in a generic direction.

	\begin{lemma}  The portrait of a Whitney umbrella is a   disjoint union of a number of simple circles  and an umbrella curve,
		 as  in Fig. \ref{fig:Whit}. Therefore the portrait  has a line symmetry. The other scenarios are non-generic. 
	\end{lemma}
	\begin{proof}  Let $x=\pi(W)$. Since we are interested in the local picture only, 
		assume that we deal with a Whitney umbrella in $\mathbb{R}^3$ and with the standard
		orthogonal  projection $\pi$  to the horizontal plane.
		Take a generic point $p\in C_x$ and a plane $h\subset \mathbb{R}^3$ which contains the endpoint $W$ of the umbrella,  and the points $p$ and $x$. Therefore $h$ contains the direction of the projection (the\textit{ vertical direction}).
		Denote by $p'\in C_x$ the point opposite to $p$.

		The intersection of a Whitney umbrella with a generic plane passing through $W$
		is  a planar curve with an $A_2$ singularity at $W$, see \cite{FukHag}. It has the following properties : 
		(1) Each vertical line in $h$ which is sufficiently close to $W$, intersects the curve in at most three points.
		(2) If $\pi^{-1}(p)$ intersects the curve at $3$ points, then $\pi^{-1}(p')$ has no intersections with the curve.
		Now a simple case analysis completes the proof.
	\end{proof}

		\begin{figure}[h]
			\includegraphics[width=0.3\linewidth]{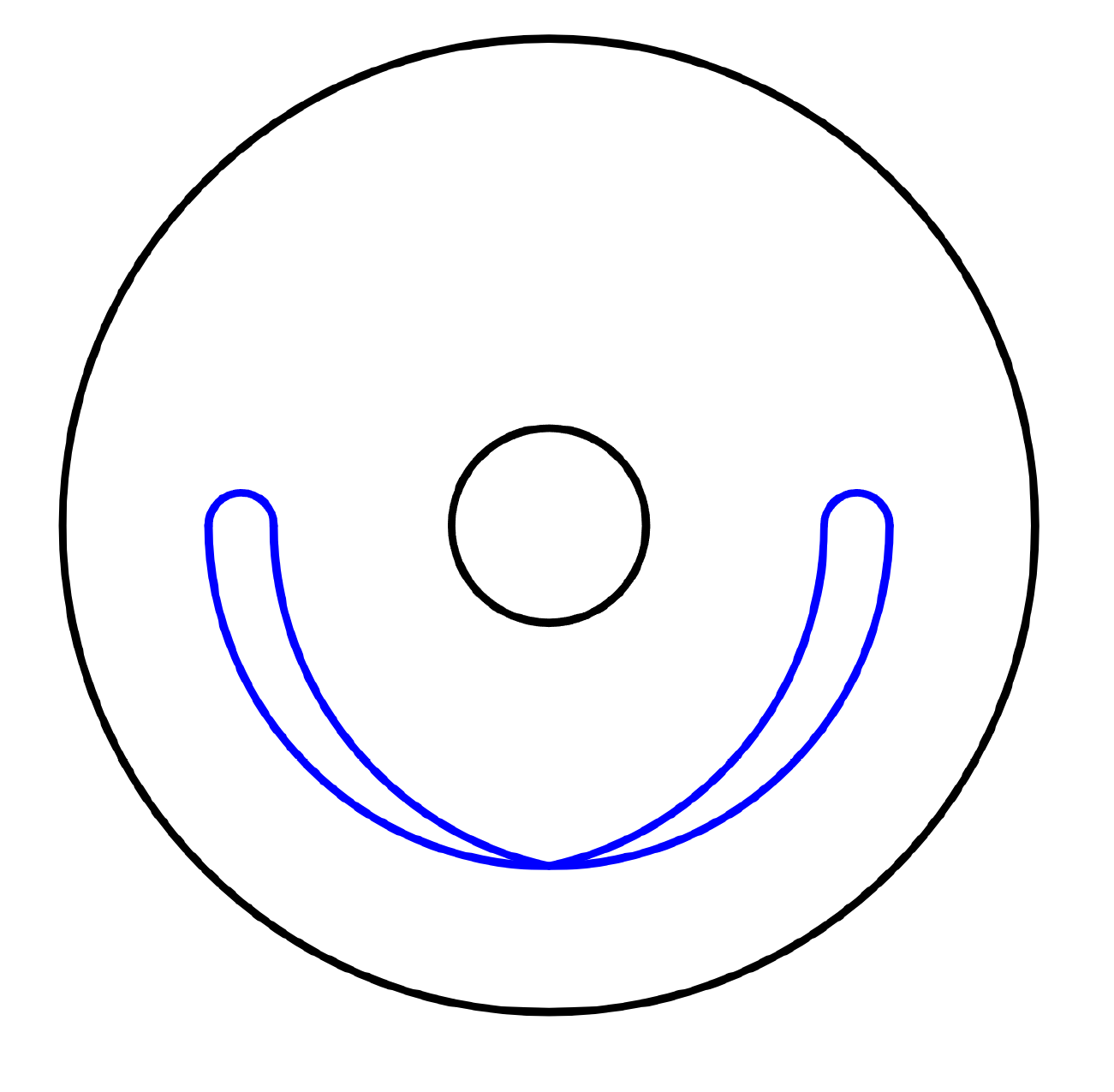}
			\caption {The portrait of a Whitney umbrella }
			\label{fig:Whit}
		\end{figure}

It is easy to check:
	\begin{lemma}  The portraits of all the singular vertices that are not listed in Lemma \ref{LemmaPortr} have a line symmetry.
	\end{lemma}


	Let $x$ be an essential vertex of type $I$.

	Imagine a point $y$ that goes along  the circle $C_x$ in the ccw direction, starting from a place with no fold  curves,
	that is, with the minimal number of points in $\pi^{-1}(y)\cap\mathcal{ Q}$.
	Let us order the folds as follows: the point $y$ meets the first fold   first. The other fold  is the second.
	
{ Now let
		$n(x)$  be the number of simple circles of  $\pi^{-1}(y)\cap \mathcal{Q}$  lying between the first and the second fold curves, 
		if one counts from the \textbf{first} fold curve in the direction of the fiber.}
	
	Set also $k(x)$ be the number of simple circles  lying between the fold curves, 
	if one counts from the \textbf{second} fold curve in the direction of the fiber.  
	
\medskip
	
	It is clear that:
	\begin{lemma}\begin{enumerate}
	               \item (Two folds intersect in the projection)\begin{enumerate}
			\item The portrait of an essential vertex of type $I$ 
			defines (and  is defined by)
			two
			numbers, $n$ and $k$.
			\item For the  line symmetry image, $n$ and $k$ interchange.
			\item  The number $n+k$ cannot be zero.
		\end{enumerate}
	               \item (A pleat)
		The  portrait of an essential vertex of type $II$    defines (and  is defined by)  its type (right or left)
		and the
		number $r$ of simple circles ($r$ might be zero).
	               \item 	(A fold intersects a regular sheet) The  portrait of an essential vertex of type $III$ defines
		(and  is defined by)  its type (right or left) and the  number $r$ of simple circles ($r$ might be zero).     
	             \end{enumerate}
\end{lemma}

\medskip

Let us assign\textit{ weights }to the essential singular vertices:

	\begin{definition}\begin{enumerate}
			\item  For an essential vertex  $x$ of type $I$, set $$\mathcal{W}_{ff}(x)=\mathcal{W}_{ff}(n,k)=\frac{4(n-k)}{(n+k)(n+k+2)(n+k+4)}.$$
			
			\item For a right essential vertex $x$ of type $II$  set $$\mathcal{W}_{p}(x)= \mathcal{W}_{p,R}(r)=\frac{2}{(r+1)(r+3)}$$
			
			For a left essential vertex $x$ of type $II$ set $$\mathcal{W}_{p}(x)=\mathcal{W}_{p,L}(r)=\frac{-2}{(r+1)(r+3)}.$$
			
			\item For a right essential vertex $x$ of type $III$, set $$\mathcal{W}_{fs}(x)=\mathcal{W}_{fs}(r,R)= \frac{2}{(r+1)(r+3)}$$
			
			For a left essential vertex $x$ of type $III$, set $$\mathcal{W}_{fs}(x)=\mathcal{W}_{fs}(r,L)=\frac{-2}{(r+1)(r+3)}.$$
			
		\end{enumerate}
		
	\end{definition}

	\newpage
	
	\section{Local formula for the Euler class}

	\begin{theorem}\label{ThmMain}  
		The Euler number of a circle bundle with a generic quasisection $\mathcal{Q}$ equals the sum of weights of essential singular vertices:
		
		$$\mathcal{E}= \sum_I \mathcal{W}_{ff}(x_i)+ \sum_{II} \mathcal{W}_{p}(x_i)+\sum_{III} \mathcal{W}_{fs}(x_i).$$

		That is, only pairwise intersections of folds, pleats, and intersections of a fold and a regular sheet matter.

	\end{theorem}

	\begin{proof}

		\textbf{{Multisections.}}
		The singular curve cuts the base $B$ into open domains $D_1,...,D_N$. For a domain  $D_i$, the restriction
		$$\pi^{-1}(D_i)\cup \mathcal{Q}\rightarrow D_i$$ is a (non-ramified) covering. If $D_i$  is one-connected, the covering is trivial.
		So let us first assume that all the  $D_i$  are one-connected. Denote the covering degree by $d_i$.
		
		We are going to construct a  random partial sections, or, equivalently, a finite number of partial sections that have equal probabilities. Each of the partial  sections will be defined everywhere except for the set 
		of singular vertices  $\{x_i\}$. 
		
		\begin{description}
			\item[Step 1] For each of $D_i$ we choose the section to be equal to one of the sheets of the covering. The choice is done independently for different domains, and uniformly for the sheets over  a given $D_i$. This gives us $\prod d_i$ partial sections defined so far over $\bigcup D_i$. Each of them comes  with probability  $\frac{1}{\prod d_i}$. 
			\item[Step 2] Extend each of the chosen sections to $B\setminus \{x_i\}$. It remains to extend a section defined over two neighbor domains $D_i$ and $D_j$  (it might be one and the same domain) to a separating segment of the singular curve.
			
			For each of the singular edges we choose the extensions according to the following rules. Assume the singular edge separates domains $D_i$ and $D_j$.
			\begin{enumerate}
				\item If the sheets over domains $D_i$ and $D_j$ agree (that is, one is the extension of the other), 
				then we extend the section to the singular edge by this sheet. 
				\item If the sheets over domains $D_i$ and $D_j$ do not  agree, we extend the section to the singular edge in two ways, clockwise and counterclockwise, see Fig. \ref{FigExtensionEdges}, with probability $1/2$ for each choice.

				Let us comment on Fig. \ref{FigExtensionEdges}. It depicts two neighbor domains, $D_i$ and $D_j$. Assume that a line $l$ crosses transversally
				the singular edge. One sees the restriction of the bundle to $l$, so the restricted total space is the cylinder. The dotted lines denote the quasisection. The red lines on the left indicate a choice of the partial sections over $D_i$ and $D_j$; in this particular case they do not agree. The  red lines on the right show two ways of extending the partial section to the singular edge. So the number of partial sections increases after Step 2 even further.
			\end{enumerate}

		\end{description}

		\begin{figure}[h]
			 \includegraphics[width=0.6\linewidth]{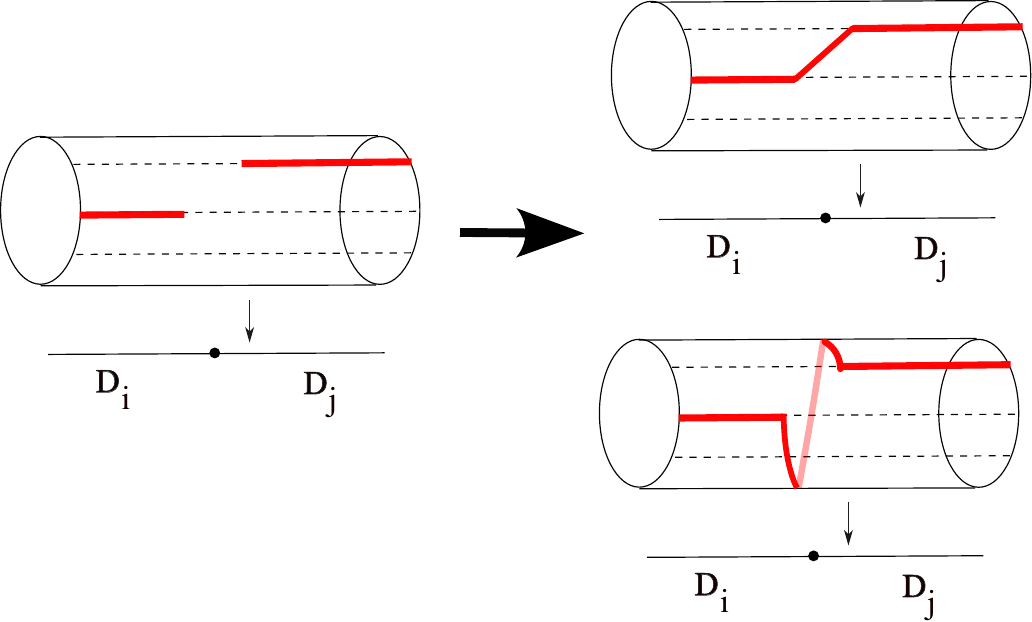}
			\caption {Extension of a partial section to a  singular edge.}
			\label{FigExtensionEdges}
		\end{figure}

		Since 
		$\mathcal{E}=\mathbb{E}\Big(\sum_i ind_s(x_i)\Big)= \sum_i \mathbb{E} (ind_s(x_i)),$
		it remains to compute the contribution $\mathbb{E} (ind_s(x))$ of each of the singular vertices.

		Let us first give technical backgrounds of the computation. In a neighborhood of a singular vertex, a choice of a partial section over the domains $D_i$ that are incident to the vertex,  amounts to 
		a  choice of a (possibly non-continuous) curve from the portrait  which intersects each of the fibers exactly once, see Fig.  \ref{FigJumps}.
		
		Assume one goes in the ccw direction. At the points of discontinuity the curve makes $M$ jumps ''forwards'', that is, in the direction of the fiber,
		and $K$ jumps backwards.

		\begin{figure}[h]
			\includegraphics[width=0.4\linewidth]{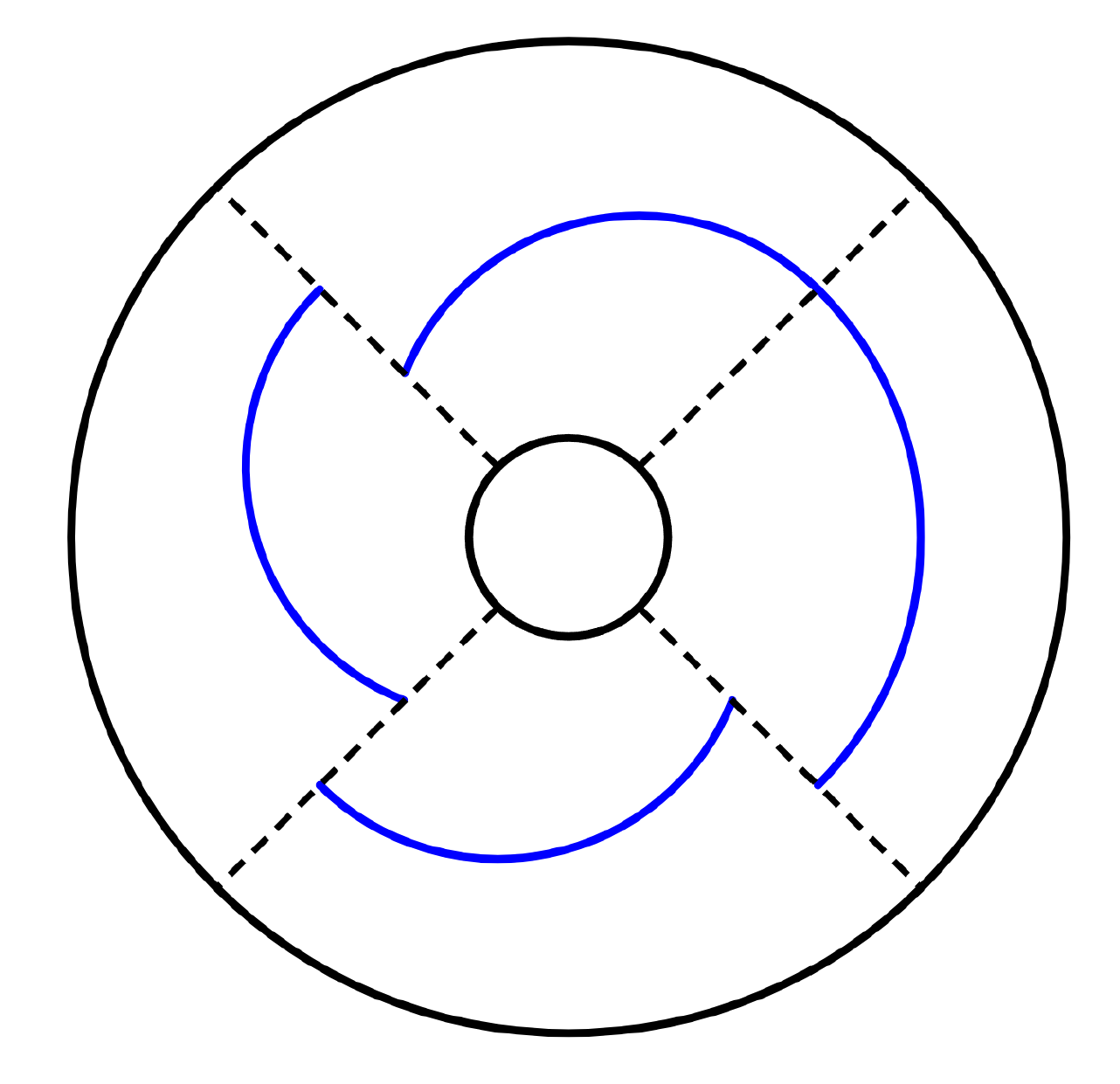}
			\caption {A choice of partial sections over the domains $D_i$ in a neighborhood of a singular vertex. Here $M=0,  \ K=3$ }
			\label{FigJumps}
		\end{figure}

		A straightforward computation shows:
		\begin{lemma}\label{LemmaJumps}
			In the above notation, a  choice of partial sections over $D_i$ contributes $\frac{K-M}{2}$.

Therefore $\mathbb{E} (ind_s(x))$ equals the average of $\frac{K-M}{2}$  over all possible choices of partial sections over the domains $D_i$ incident to $x$.
		
		\end{lemma}

		\begin{lemma} \label{LemmaSymm} \begin{enumerate}
				\item The contributions of a portrait and its image under a line symmetry sum up to zero.
				\item If a portrait of a singular vertex has a line symmetry, the vertex contributes $0$.
\item Only essential singular vertices have non-zero expectation $\mathbb{E} (ind_s(x))$.
			\end{enumerate}
			
		\end{lemma}
		\begin{proof}

			(1)  Indeed, {line symmetry maps a partial section $s$ of an initial portrait to some  partial sections of the symmetric image of the portrait. Besides, $K$ and $M$ exchange their roles. }
			(2) {follows from (1), since the numbers $\frac{K-M}{2}$ for two symmetric partial sections sum up to zero.}
			
		\end{proof}

\begin{lemma}
			For a right essential  vertex of type $II$ (a right pleat), $$ \mathbb{E} (ind_s(x))=\frac{2}{(r+1)(r+3)}.$$
			
			For a left essential  vertex of type $II$ (a left pleat), $$ \mathbb{E} (ind_s(x))=\frac{-2}{(r+1)(r+3)}.$$
			
		\end{lemma}  
		\begin{proof}
			
			{ Let's prove the statement for a right pleat. The singular vertex has two incident domains $D_i$ and $D_j$, so } {the} portrait is divided into two regions: a large {sector} and a small sector, see Fig. \ref{FigPlDivided}. Now, using {Lemma \ref{LemmaJumps}}, let us calculate the contribution of each of the partial sections. If we choose a red arc in the large sector and a red arc in the small sector, we get either two jumps down, as in case (a), or one jump down, as in cases (b) and (c). Their total contribution is $\frac{2}{2} + 2\cdot\frac{1}{2} = 2$. If we choose {a } blue arc in the small sector, we will have one downward and one upward jump,  {so}  the total contribution will be zero. Also, {a choice of any blue arc in the large sector and any other arc in the small sector  contributes  zero}. {Since the total number of sections is $(r + 1)(r + 3)$, the averaging gives  $\frac{2}{(r + 1)(r + 3)}$.}

	\begin{figure}[h]
		\includegraphics[width=0.9\linewidth]{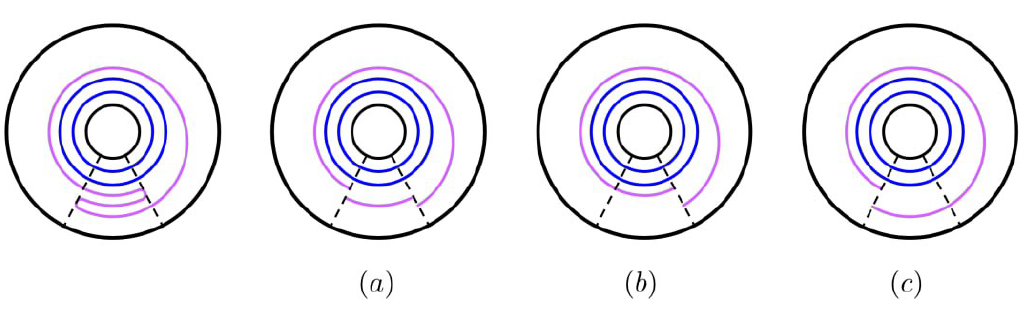}
		\caption {portrait of a pleat; choices of partial sections.}
		\label{FigPlDivided}
	\end{figure}

		\end{proof}

		\begin{lemma}
			For an essential vertex of type $I$,  $$\mathbb{E} (ind_s(x))=\frac{4(
				{n}-{k})}{(n+k)(n+k+2)(n+k+4)}$$

		\end{lemma}  
		\begin{proof}\footnote{A computation is contained in \cite{Kaz}. For the sake of completeness, we also give a proof here.}

\begin{figure}[h]
			\includegraphics[width=0.3\linewidth]{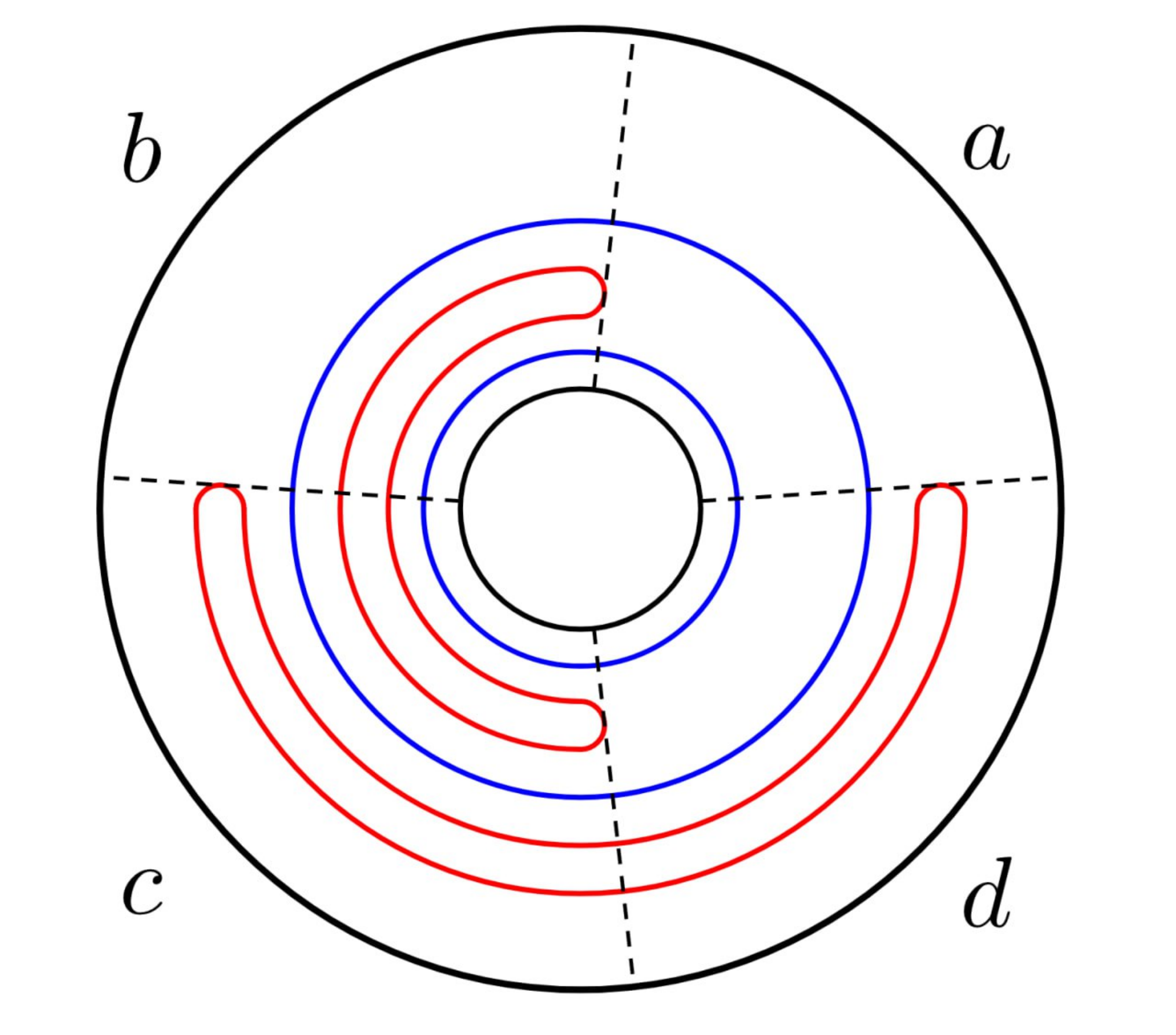}
			\caption {The four domains $D_i$ in a neighborhood of a  type $I$ critical vertex.}
			\label{FigMaxProof}
		\end{figure}
 The annulus is divided into four sectors, $a,b,c$ and $d$, see \ref{FigMaxProof}. the sectors come from  four domains $D_i$ in a neighborhood of the vertex. By Lemma \ref{LemmaJumps}, we have to average  $\frac{K-M}{2}$ over all choices of sheets  in each of the sectors.  A $4$-tuple of partial sections (over $a,b,c$ and $d$) yields two triples of partial sections, one triple is over the sectors $a,b,c$, and the other one is over  $a,d,c$. Observe that the value $\frac{K-M}{2}$ behaves additively. Therefore  we need to average $\frac{K-M}{2}$ for both of the triples and sum up. A  computation (analogous to the previous lemma) shows that each triple contributes $\frac{2(n-k)}{(n+k)(n+k+2)(n+k+4)}$.

		\end{proof}

		\begin{lemma}
			For a right essential  vertex of type $III$ (a regular sheet intersects a fold),$$\mathbb{E} (ind_s(x))= \frac{2}{(r+1)(r+3)}.$$
			
			For a left essential  vertex of type $III$, $$\mathbb{E} (ind_s(x))= \frac{-2}{(r+1)(r+3)}.$$
			
		\end{lemma}   
		\begin{proof}
			This can be proven by analogous computations.
		\end{proof}

		Theorem is proven for the case when all the domains $D_i$ are one-connected.
		
		\textbf{Non-one-connected domains $D_i$.}
		
		If there are  domains that are not one-connected, the following trick helps.
		Let us enlarge $\mathcal{Q}$  by adding a  disjoint with $\mathcal{Q}$ connected component.  This new component is a topological sphere.
		Geometrically it looks like a thin ellipsoid (Hausdorff close to a line segment).
		On the one hand, a careful  adding  a number of such components makes all the $D_i$ one-connected, keeps the Euler number, and also keeps the portraits of
		already existing singular vertices.
		On the other hand, there  arise new singular vertices of type $I$. They arise in pairs  that differ on a line symmetry, 
		therefore the corresponding weights are canceled.

	\end{proof}

	\section{Examples of quasisections. Existence and non-existence}\label{SecExamples} 
	
	\subsection{Elementary examples}
	A smooth section  is a quasisection. Indeed, it can be viewed as an embedded closed manifold which projects bijectively to the base $B$.
	We remind the reader that existence of a continuous section  is equivalent to triviality of the circle bundle, that is, to $\mathcal{E}=0$.

	Moreover, if  a quasisection $\mathcal{Q}$  has no singularities (for instance, yields a covering of $B$), 
	then $E\rightarrow B$ is also trivial.

	One can fill $E$ with many mutually disjoint (topological) spheres  and thus get a (disconnected) quasisection with type $I$ singular vertices only.
	Adding  thin connecting tubes  gives a connected embedded quasisection.

	\begin{example} \label{ExTriv} (Quasisections of the trivial bundle.)
		Let   $\phi:Q\rightarrow B$ be a generic smooth surjective map.   Let $f:Q\rightarrow \mathbb{R}$  be a
		generic function  such that  $f(Q)\subset (0,1)$. Assume that the fibers of the bundle are consistently coordinatized by $\mathbb{R}/\mathbb{Z}$.
		Set $q(x)=(\phi(x), f(x))$. Adding $N$ disjoint continuous sections  gives a quasisection.  If $\pi\circ q$ is surjective, the number $N$ may be zero.
		
	\end{example}

	\subsection{Quasisections with odd degree}
	\begin{definition}
		The degree $deg(\mathcal{Q})\in \mathbb{Z}_2$ (or $\mathbb{Z}$) of a quasisection is $deg(\pi \circ q)$. 
	\end{definition}

	\begin{theorem}
		\begin{enumerate}
			\item If $Q$ is orientable,  and $\mathcal{E}\neq 0$, then $deg(\mathcal{Q})=0\in \mathbb{Z}$.
			\item If $Q$ is non-orientable, and $deg(\mathcal{Q})=1\in \mathbb{Z}_2$, then $\mathcal{E}$  is even.
			\item For any even $\mathcal{E}$, there exists  a non-orientable quasisection with $deg=1$.
		\end{enumerate}
	\end{theorem}
	
	\begin{proof}
		(1)  follows from Gysin exact sequence.

Indeed, let us interpret $\mathcal{Q}$ as  Poincar\`{e}  dual of a cohomology class $\mathbf{q}$. Denote  by $[B]\in H^0(B, \mathbb{Z})$ is the Poincar\`{e} dual of the fundamental class of the base. The part of the exact sequence we need is:
$$ H^1(E,\mathbb{Z})\xrightarrow[\text{}]{\pi_*} H^0(B,\mathbb{Z}) \xrightarrow[\text{}]{\smile\mathcal{ E}} H^0(B,\mathbb{Z}). $$

The Gysin homomorphism sends $\mathcal{Q}$ to $deg(\mathcal{Q})\cdot [B]$. The cup-product with the Euler class amounts to multiplication by Euler number.
	Since the composition	$\pi_*(\mathbf{q})\smile \mathcal{E}= deg(\mathcal{Q})\cdot \mathcal{E}\cdot [B]$   vanishes,  either $\mathcal{E}$ or $deg(\mathcal{Q})$ vanish, which proves (1).  The statement (2) follows analogously from the Gysin sequence with $\mathbb{Z}_2$ coefficients.
		
		An example for  (3) is below.
	\end{proof}
	
	\begin{example}\label{Exodd}
		Assume $\mathcal{E}(E\rightarrow B)=2$. Take the base $B$ and a small disk $D\subset B$. The restriction of the bundle to $B\setminus Int(D)$ is trivial,
		so let us choose a continuous section $s$ over $B\setminus Int(D)$. One  considers $s$ as a bordered surface $S\subset E$. By construction, $\partial S\subset \pi^{-1}(\partial D)=\partial D\times S^1$.  We may assume that $\partial S$ looks as  in Fig. \ref{FigExOdd}, left. Homotope $S$ as is depicted in Fig. \ref{FigExOdd}, center. Next patch two discs to  $\partial S$ (Fig. \ref{FigExOdd}, right). One disc projects to $D$, and the
		other one embeds in $\partial D\times S^1$. We get a closed surface (one can check that it is the Klein bottle) which is piecewise smoothly mapped to $E$.
		A  perturbation turns it to a smooth  quasisection.

 Our surgery acts in a neighborhood of $D$. Therefore the degree of $\pi\circ q$ equals $1$ since  each of the points away from $D$ has a one-element preimage.

A quasisection for arbitrary  even $\mathcal{E}=2k$ comes from an analogous construction for $k$ small disks $D_i \subset B$.
	\end{example}
	
	\begin{figure}[h]
		\includegraphics[width=0.5\linewidth]{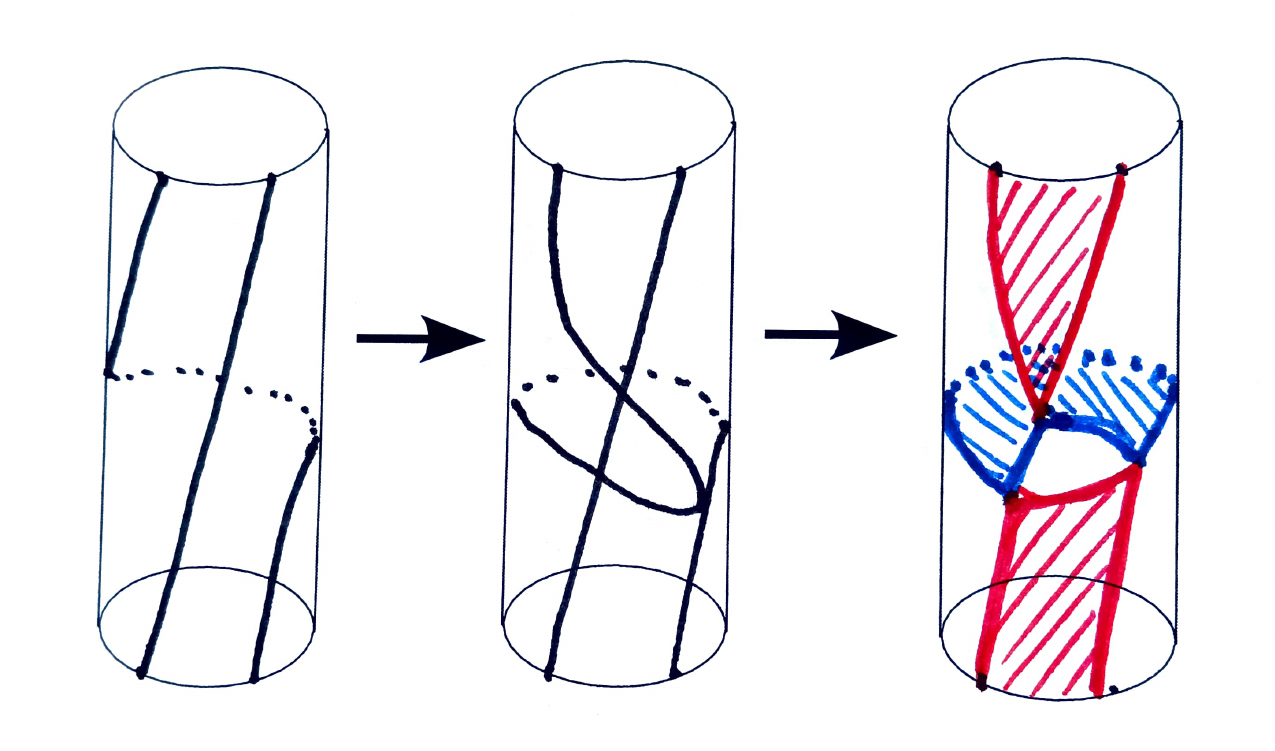}
		\caption {A quasisection with $deg=1$.}
		\label{FigExOdd}
	\end{figure}

	\subsection{Pancake quasisections}
	
	A \textit{pancake quasisection}  
	is an embedding of a finite number of topological spheres in $E$ such that:
	\begin{enumerate}
		\item Each of the spheres projects with no pleats and with a unique circular fold.
		\item The projections of the spheres 
		have at most triple intersections.
	\end{enumerate}
	
	The spheres  are called  \textit{ pancakes}   since one imagines  them to be very much flattened.  
	
	\medskip
	
	Assume we have a pancake quasisection, and $P_1,P_2$ and $P_3$ are pancakes that intersect in the projection.
	The triple produces three singular vertices  (the innermost intersection points)  of   type $I$. Each of them has weight either $1/6$, or $-1/6$, depending on the orientation of the triple $P_1,P_2, P_3$.  Therefore the local formula reads as
	$$\mathcal{E}(E\rightarrow B)=\frac{1}{2}(\hbox{ number of positively oriented triples of pancakes } -$$ $$\hbox{ number of negatively oriented triples}).$$

	\begin{figure}[h]
		\includegraphics[width=0.4\linewidth]{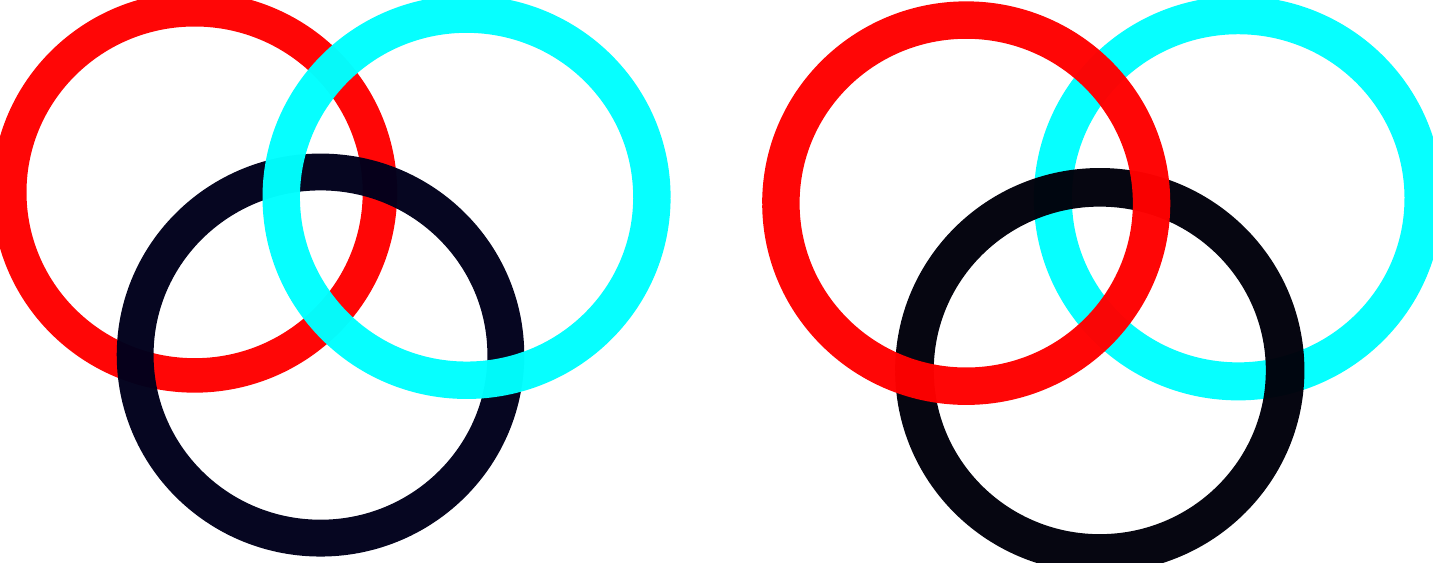}
		\caption {Two triples of pancakes, a negatively oriented (left) and a positively oriented  (right).}
		\label{FigExtensionEdges}
	\end{figure}

	\begin{example}\label{Ex4pancakes}
		The circle bundle over the sphere $S^2$ with $\mathcal{E}=2$ has a pancake quasisection with four pancakes.
		All the $12$ singular vertices are essential and have one and the same type $I$ with $n=2$, $k=0$.
	\end{example}

	\begin{example}\label{ExCrossingPanA}(Two crossing pancakes, A)

		Let  $E\rightarrow B$ be a circle bundle with $\mathcal{E}=1$ over $B=S^2$. Let $U\subset B$ be a disk. The bundle has a continuous section $s_1$ over $U$,
		and a continuous section $s_2$ over $B \setminus U$.  We think of these partial sections as of subsets of $E$. Take the $\varepsilon$-neighborhoods
		of $s_1$ and $s_2$ and their boundaries. We get two ''very flat'' topological spheres (also called pancakes), but unlike the previous example, the pancakes are now crossing. The only essential vertices in this case are $4$ right essential vertices of type $III$ with $r=1$.  

If $\mathcal{E}\geq 0$  is arbitrary, a similar construction gives $4\mathcal{E}$ right essential vertices of type $III$ with $r=1$
	\end{example}
      
\begin{example}\label{ExCrossingPanB}(Two crossing pancakes, B)

		Let  $E\rightarrow B$ be a trivial bundle over $B=S^2$. Take $n$ continuous disjoint sections and two crossing discs in $E$.  Replacing the discs by their $\varepsilon$-neighborhoods gives two crossing pancakes.  We arrive at a quasisection with $4$ singular vertices of type $III$, and $2$ singular vertices of type $I$.

	\end{example}
	\begin{example}\label{CurledPancake} (Curled pancake)
		 Consider a trivial bundle over $S^2$. Let $\mathcal{Q}$ consist of  $N$ disjoint continuous sections one curled pancake, see Fig. \ref{FigCurled}. There are two essential vertices of type $II$ and one of type $I$. 
	\end{example}

\begin{figure}[h]
		\includegraphics[width=0.3\linewidth]{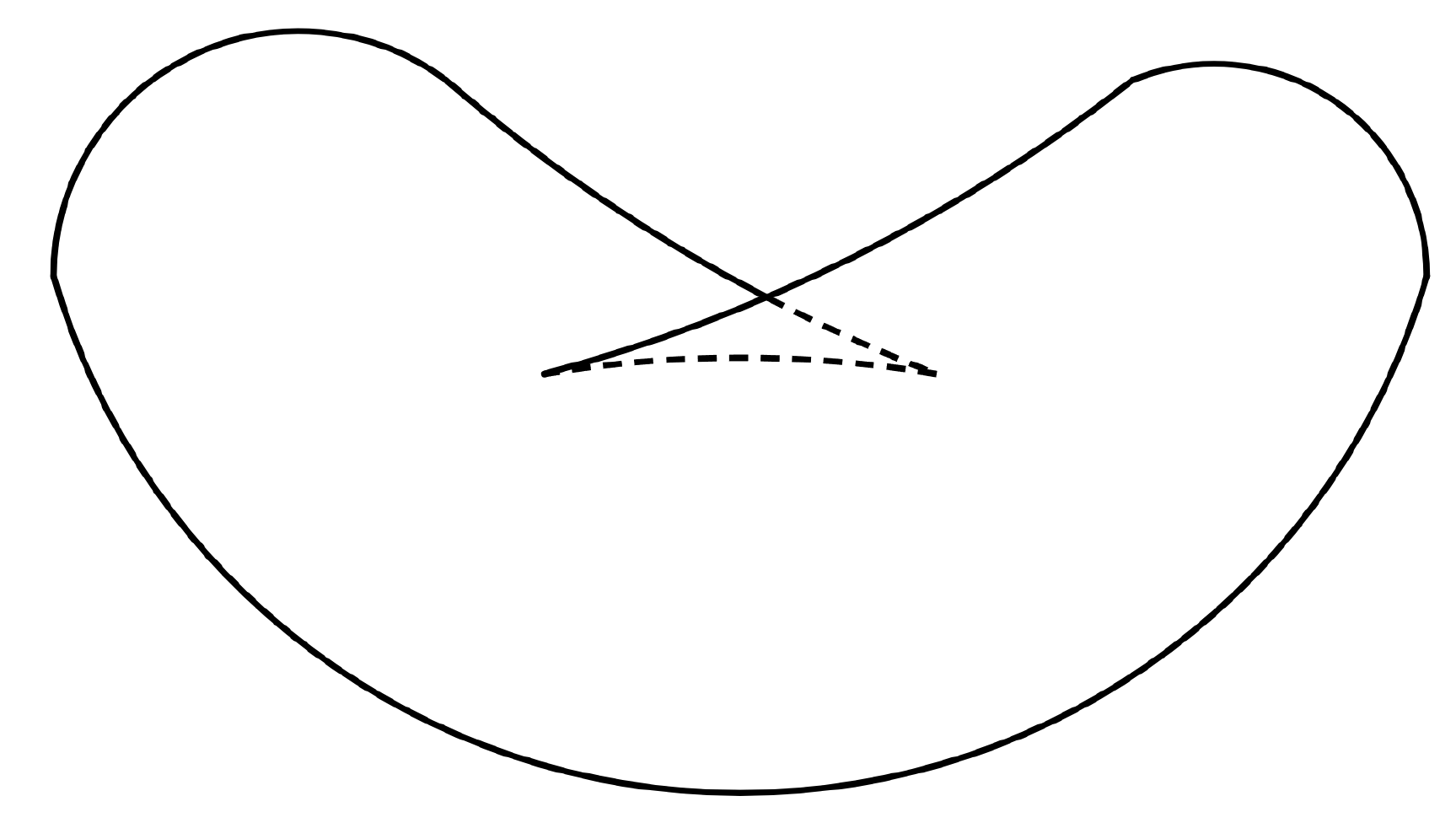}
		\caption {A curled pancake}
		\label{FigCurled}
	\end{figure}

	\section{Uniqueness of the local formula}  \label{SecUni}
	
	\begin{definition}
		A \textit{local formula $\mathcal{F}$  
computing the Euler number via singularities of quasisections} is a triple of functions
		$$\mathcal{F}_{ff}: (\mathbb{N}\cup \{0\})\times (\mathbb{N}\cup \{0\})\setminus \{0,0\}\rightarrow\mathbb{ R},$$  $$\ \mathcal{F}_p:(\mathbb{ N}\cup \{0\})\times \{L,R\}\rightarrow\mathbb{ R}, $$ $$
		\ \hbox{and}\ \ \mathcal{F}_{fs}: (\mathbb{ N}\cup \{0\})\times \{L,R\}\rightarrow\mathbb{ R}$$
		
		such that for any circle bundle $E\rightarrow B$ and any quasisection $\mathcal{Q}$
		$$\mathcal{E}(E\rightarrow B)= \sum_I \mathcal{F}_{ff}((n,k)(x_i))+ \sum_{II} \mathcal{F}_{p}(n(x_i))+\sum_{III} \mathcal{F}_{fs}(n(x_i)),$$
		
		where $x_i$ ranges over essential vertices of $\mathcal{Q}$, and the numbers $n$  (and $k$ for the type $I$) are those defined in Section \ref{SecPortraits}.
	\end{definition}
	
	\begin{theorem}\label{ThmUniqueness}
		Any local formula  coincides  with that from Theorem \ref{ThmMain}.
		That is, we always have
		$$\mathcal{F}_{ff}(n,k)\equiv \mathcal{W}_{ff}(n,k), \ \mathcal{F}_{p}(n)\equiv \mathcal{W}_{p}(n), \ \mathcal{F}_{fs}(n)\equiv \mathcal{W}_{fs}(n).$$
	\end{theorem}

	The proof is based on examples from Section \ref{SecExamples}  and goes by induction.

	We apply local transformations to a quasisection: add pancakes or other connected components, create wrinkles, etc. These transformations keep the bundle, and therefore, the Euler number, but change the  essential singular vertices and their weights in a controllable way.
	
	Since local transformations keep the degree of the quasisection, we need two bases of induction, one for $deg=1$ and the other for $deg=0$.

\medskip

So assume we are given some local formula $\mathcal{F}$.

	\begin{lemma} \label{LemmaAntisymm}(The antisymmetry property of $\mathcal{F}$) \begin{enumerate}
			
			\item $\mathcal{F}_{ff}(n,k)\equiv -\mathcal{F}_{ff}(k,n)$.
			\item $\mathcal{F}_p(n,R)=-\mathcal{F}_p(n,L)$.
			\item $\mathcal{F}_{fs}(n,R)\equiv -\mathcal{F}_{fs}(n,L)$.
			
		\end{enumerate}
		
	\end{lemma} 
	
	\begin{proof}
\begin{enumerate}
  \item Take a quasisection $\mathcal{Q}$  and add two small disjoint pancakes,  whose projections  intersect at two points. We assume that the pancakes do not intersect $\mathcal{Q}$, and the projections of  two new folds do not intersect the already existent singular curves. There arise two new singular vertices of type $I$, one with some parameters $(k,n)$, and the other with $(n,k)$. Since all other singular vertices and their weights  persist, and  $\mathcal{E}$  persists as well, the  weights of the new singular vertices  cancel each other.
  
  \item Take a  regular part of a quasisection $\mathcal{Q}$  and create a small wrinkle, that is, two pleats connected with folds. The parameters of the pleats are $r,L$ and $r,R$.
  Since all other singular vertices persist, and  $\mathcal{E}$  persists as well, the weights of the two pleats  cancel each other.
  \item  Take a regular point of $\mathcal{Q}$ and add a pancake that crosses $\mathcal{Q}$ in the small neighborhood of this point. There appear two new singular vertices of type $III$, one is $(n, R)$, and the other is $(n, L)$.   By similar reasons their weights   cancel each other.
\end{enumerate}
		
	\end{proof}

	\begin{lemma}\label{LemmaInduction1}(Induction step for Type $I$)  
	
			  $\mathcal{F}_{ff}(n,k)$  is uniquely determined by  $\mathcal{F}_{ff}(1,0),\mathcal{F}_{ff}(2,0),$ and $\mathcal{F}_{ff}(3,0)$. 
		
	\end{lemma}
		\begin{proof}

 We shall prove that $\mathcal{F}_{ff}(n,k)$  is a linear combination of  $\mathcal{F}_{ff}(1,0),\mathcal{F}_{ff}(2,0),$ and $\mathcal{F}_{ff}(3,0)$ with coefficients depending on $n$ and $k$. The proof goes by induction. Consider a quasisection $\mathcal{Q}$ of the trivial bundle over $S^{2}$  which is a disjoint union of three disjoint pancakes $P_1,P_2,P_3$ (whose projections have a triple intersection), and $m\geq 2$ disjoint continuous sections, called \textit{simple sheets}.  The quasisection has  $6$ singular vertices, all of them of type $I$ . Assume that there are $a$ simple sheets between pancakes $P_1$ and $P_2$, $b$ simple sheets lying between $P_2$ and $P_3$, and $c=m-a-b$ simple sheets between $P_3$ and $P_1$. 

We shall use  the local formula: 
$$ \mathcal{F}_{ff}(a+b+2,c)+\mathcal{F}_{ff}(b+c+2,a)+\mathcal{F}_{ff}(c+a+2,b)+$$ $$+\mathcal{F}_{ff}(a+b,c)+\mathcal{F}_{ff}(b+c,a)+\mathcal{F}_{ff}(c+a,b)=0, $$

and also   Lemma \ref{LemmaAntisymm}: $$\mathcal{F}_{ff}(s,m+2-s)=-\mathcal{F}_{ff}(m+2-s,s).$$
  
  
 Assuming   that $\mathcal{F}_{ff}(a+b,c)+\mathcal{F}_{ff}(b+c,a)+\mathcal{F}_{ff}(c+a,b)$ for all $a+b+c=m$ is known, prove  that
$\mathcal{F}_{ff}(n,k)$  is uniquely defined for $n+k=m+2$.

			For a fixed $m$, we arrive at a  system of linear equations  with unknowns \\ $\mathcal{F}_{ff}(m+2-a,a), \ \ a=0,1,...,m+2$.  To prove the uniqueness of its solution, let us consider the  corresponding homogenous system.
			It reads as $$\left\{
             \begin{array}{ll}
              g(a)+g(b)+g(c)=0 , & \hbox{  for $a+b+c=m$;} \\
               g(a)+g(m+2-a)=0, & \hbox{for $0 \leq a \leq m+2$.}
             \end{array}
           \right.$$

			Since $g(0)+g(0)+g(m)=0,$ and $ g(2)+g(n)=0$, we get $g(2)=2g(0)$. Since $g(1)+g(1)+g(n-2)=g(0)+g(2)+g(n-2)$, then $2g(1)=g(0)+g(2)$, so $g(1)=\frac{3}{2}g(0)$. Proceeding in a similar way, we get  $g(k)=(\frac{k}{2}+1)g(0)$ for all $k$. The second equation
gives  $0=g(0)+g(m+2)=g(0)(2+\frac{n+2}{2})$, which implies
 $g(0)=0$ and completes the proof. \\ 
			\end{proof}

\begin{lemma}\label{LemmaInduction3}(Induction step for Type $II$)

$$\mathcal{F}_p(n+2,R)=\mathcal{F}_p(n,R)-\mathcal{F}_{ff}(n+1,0)+\mathcal{F}_{ff}(n,1)$$

	\end{lemma}
\begin{proof}
  Assume  a quasisection has a right pleat with $n$ simple  circles in the portrait. Add a small pancake  right over the pleat whose projection embraces the projection of the  pleat.  The pleat turns to a pleat with $n+2$ simple  circles in the portrait. Besides, there arise two type $I$ singular vertices. Since the Euler class persists, the difference of two local formulae is zero.
\end{proof}

	\begin{lemma} (Type $III$  reduces to Type $I$)
 $$\mathcal{F}_{fs}(n + 1,R)=1/2\mathcal{F}_{ff}(n,0).$$
	\end{lemma}   
	
	\begin{proof}
	 The local formula for Example \ref{ExCrossingPanB} reads as $2\mathcal{F}_{ff}(0, n) + 4\mathcal{F}_{fs}(n + 1, L) = 0$.
	\end{proof}
 
	\begin{lemma} (The base of induction,  degree  $0$)\begin{enumerate}
			\item $\mathcal{F}_{ff}(2,0)= \mathcal{W}_{ff}(2,0)=1/6$.

			\item $\mathcal{F}_{fs}(1,R)=\mathcal{W}_{fs}(1,R)=1/4$.
			\item $\mathcal{F}_{p}(1,R)=\mathcal{W}_{p}(1,R)=1/4$.
		\end{enumerate}
	\end{lemma}            
	\begin{proof}
		(1) follows from Example \ref{Ex4pancakes}.
		
		(2) follows from Example \ref{ExCrossingPanA}.
		
		(3) Take the Boy surface  \cite{Boy}, place it in the total space of the trivial bundle $S^2\times S^1 \rightarrow S^2$, and add two continuous sections, disjoint from the Boy surface.
		We get a quasisection with degree $2$.   The local formula $\mathcal{F}$  reads as
		
$$0=3\mathcal{F}_{p}(3,R)+3\mathcal{F}_{fs}(3,R)-3\mathcal{F}_{ff}(2,0).$$ Applying Lemma \ref{LemmaInduction3}, (1) and  (2),  one gets
	
$$\mathcal{F}_p(1,R)=1/4=\mathcal{W}_p(1,R).$$

	\end{proof}

\begin{lemma} $$\mathcal{F}_{ff}(3, 0)=\frac{3}{7}\mathcal{F}_{ff}(1, 0), \ \ \mathcal{F}_{ff}(2, 1)=\frac{1}{7}\mathcal{F}_{ff}(1, 0). $$

	\end{lemma}  

\begin{proof}  Let us prove that \begin{equation*}
 \begin{cases}
  2\mathcal{F}_{ff}(3, 0) + \mathcal{F}_{ff}(0, 1) + \mathcal{F}_{ff}(2, 1) = 0
   \\
  \mathcal{F}_{ff}(0, 1) + \mathcal{F}_{ff}(3, 0) + 4\mathcal{F}_{ff}(2, 1) = 0
   \\
   
 \end{cases}
\end{equation*}
   Take the trivial bundle  $S^2\times S^1\rightarrow S^2$ with one continuous section and three pancakes as in Fig. \ref{FigExtensionEdges}. The local formula gives the first equation.

   For the second equation let us take a trivial bundle over $S^2$ with one curled pancake and one small pancake placed over one of the pleats.The local formula gives: $\mathcal{F}_{p}(2, L) + \mathcal{F}_{p}(4, L) + \mathcal{F}_{ff}(0, 1) + \mathcal{F}_{ff}(2, 1) + \mathcal{F}_{ff}(2, 1) = 0$. Application of the antisymmetry property and Lemma \ref{LemmaInduction3} complete the proof.
\end{proof}

	\begin{lemma} (The base of induction, degree $1$)\begin{enumerate}
			\item $\mathcal{F}_{ff}(1,0)= \mathcal{W}_{ff}(1,0)=4/15$.
			\item $\mathcal{F}_{p}(0,R)=\mathcal{W}_{p}(0,R)=2/3$.
			\item $\mathcal{F}_{fs}(0,R)=\mathcal{W}_{fs}(0,R)=2/3$.
			
		\end{enumerate}
	\end{lemma}            
	
	\begin{proof}
		Let us take the Boy surface, place it in the total space of the trivial bundle $S^2\times S^1 \rightarrow S^2$, and add a continuous section.
		The union of the Boy surface and the continuous section is a quasisection with degree $1$. Application of the local formula $\mathcal{F}$
and Lemma \ref{LemmaInduction1}  give: $$0=3\mathcal{F}_{p}(2,R)+3\mathcal{F}_{fs}(2,R)-3\mathcal{F}_{ff}(1,0)=$$
	$$=3(\mathcal{F}_p(0,R)-\mathcal{F}_{ff}(1,0)+\mathcal{F}_{ff}(0,1)) +3/2\mathcal{F}_{ff}(1,0)-	3\mathcal{F}_{ff}(1,0)=$$
$$= 3\mathcal{F}_p(0,R)-3/2\mathcal{F}_{ff}(1,0).$$
		We get a homogeneous linear equation which includes unknowns $\mathcal{F}_{p}(0,R)$ and $\mathcal{F}_{ff}(1,0)$.
		
		Next, we take the quasisection from Example \ref{Exodd}, where $deg=1$, and  $\mathcal{E}=2$. Applying the local formula $\mathcal{F}$ together with the above lemmata, we get
		$$2=\lambda\mathcal{F}_{p}(0,R)+\mu\mathcal{F}_{ff}(1,0)$$  for some coefficients $\lambda$ and $\mu$.  Thus we obtain a system of linear equations which has a unique solution. 
		
	\end{proof}

Theorem \ref{ThmUniqueness} is proven.

\begin{remark}\label{remark}
  Theorem \ref{ThmMain} can be proven based on the approach of the section. Firstly,  a local formula exists, this follows from
the general averaging principle. Secondly, the values of the weights are retrieved uniquely, this is what has been done in the section.
\end{remark}

	\section{Local formula for bordered quasisections}\label{SecBordered}

\begin{definition}
		A bordered quasisection of $ E \xrightarrow[\text{}]{\pi} B$ is a  smooth surface  $Q$  (possibly with boundary) and a smooth map
		$q:Q\rightarrow E$
		such that $\pi\circ q (Q)=B$. 
	\end{definition}

	With a bordered quasisection there arise new types of essential vertices  (in addition to the ''old'' types $I$, $II$ and $III$):

		\begin{enumerate}
			
			\item  \textbf{Type $I_{bb}$}  is an intersection point of (projections of) two boundary curves. Its portrait is depicted   in Fig. \ref{fig:2nonsymm}, left, up to adding/removing  some of simple circles.

			\item  \textbf{Type $I_{bf}$, right} is an intersection point of (projections of) a boundary curve  and a fold is  as  in Fig. \ref{fig:2nonsymm}, center, up to adding/removing  some of simple circles.

			The portrait of \textbf{Type $I_{bf}$, left}  is a line image of a  portrait of \textbf{Type $I_{bf}$, right.}

			\item \textbf{ Type $III_{bs}$, right} (respectively, \textbf{left}) is the projection of an intersection of a boundary curve and  a regular sheet. Its portrait  is as in Fig.  \ref{fig:2nonsymm}, right. The portrait of \textbf{ Type $III_{bs}$, left} is line symmetric to the figure).

			\item  \textbf{ Type $IV_{tf},$ right(left)}  arises as  the projection of a point where a fold terminates on the border line, see Fig. \ref{FigTermFold}. 
Their portraits are in Fig. \ref{fig:3nonsymm}
		\end{enumerate}

	\begin{figure}[h]
		\includegraphics[width=0.8\linewidth]{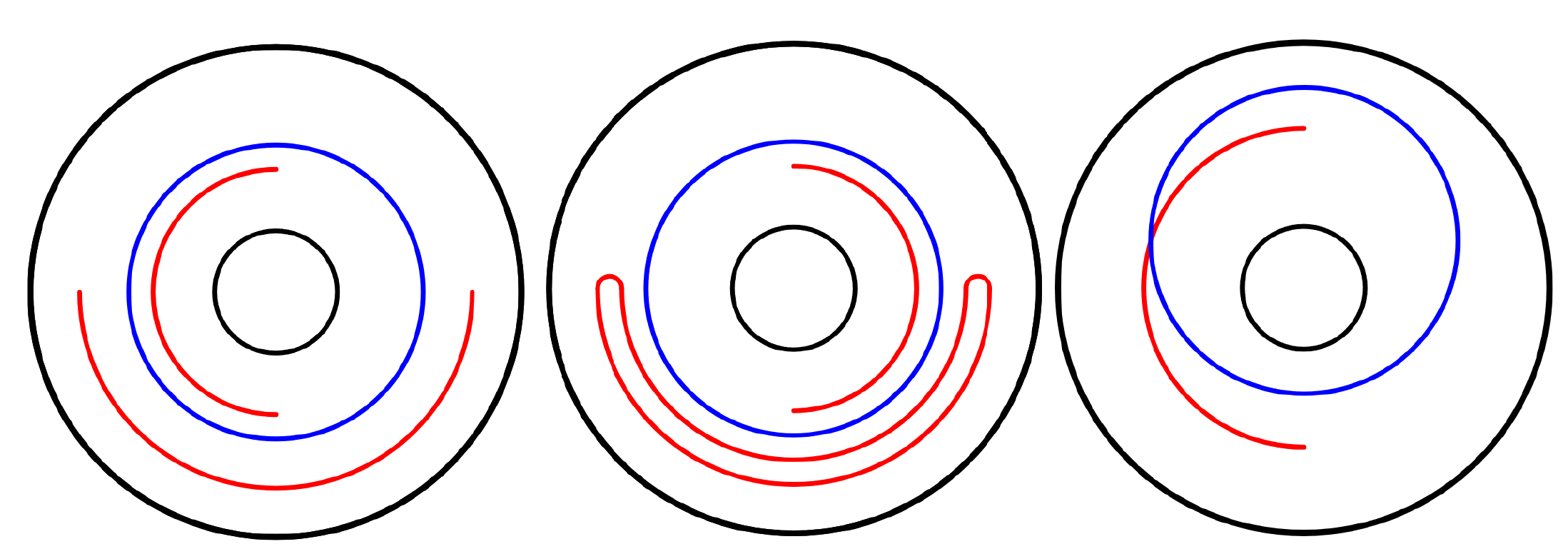}
		\caption {Essential vertices for bordered quasisection, types $I_{bb}$, $I_{bf}$  right, and $III_{bs}$  right. }
		\label{fig:2nonsymm}
	\end{figure}

	\begin{figure}[h]
		\includegraphics[width=0.3\linewidth]{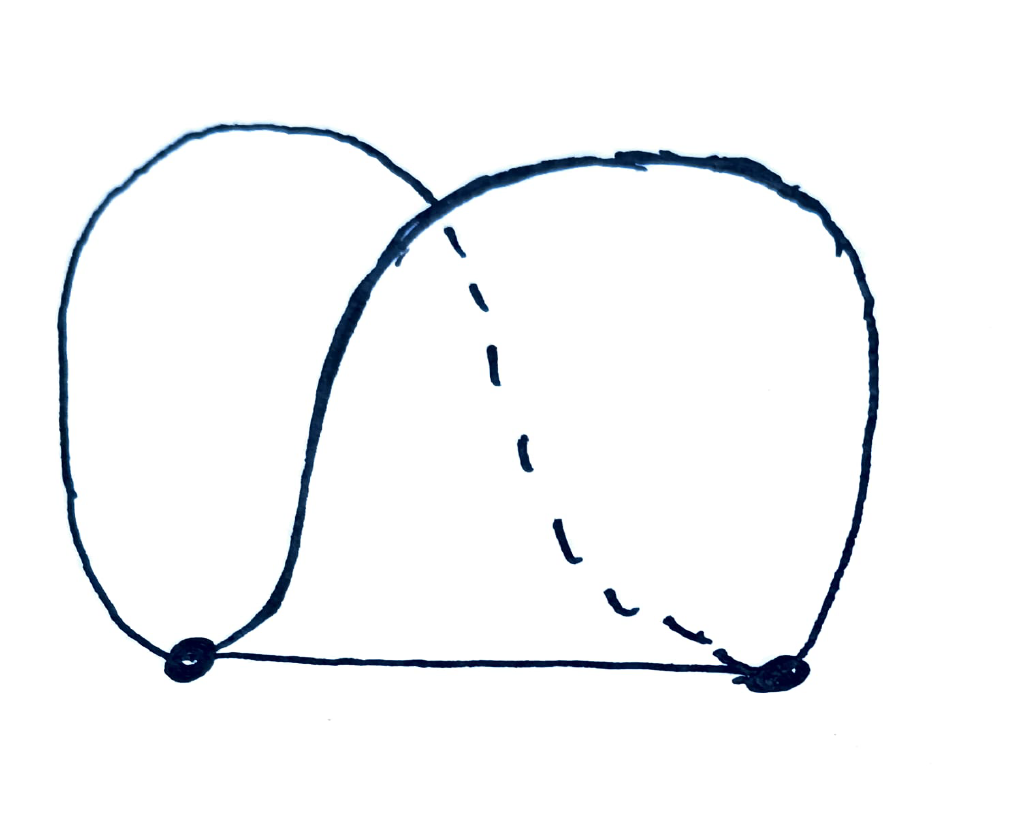}
		\caption {A curled disk  gives two singularities of  Type $IV_{tf},$ left, and one singularity of type Type $I_{bb}$. }
		\label{FigTermFold}
	\end{figure}

	\begin{figure}[h]
		\includegraphics[width=1\linewidth]{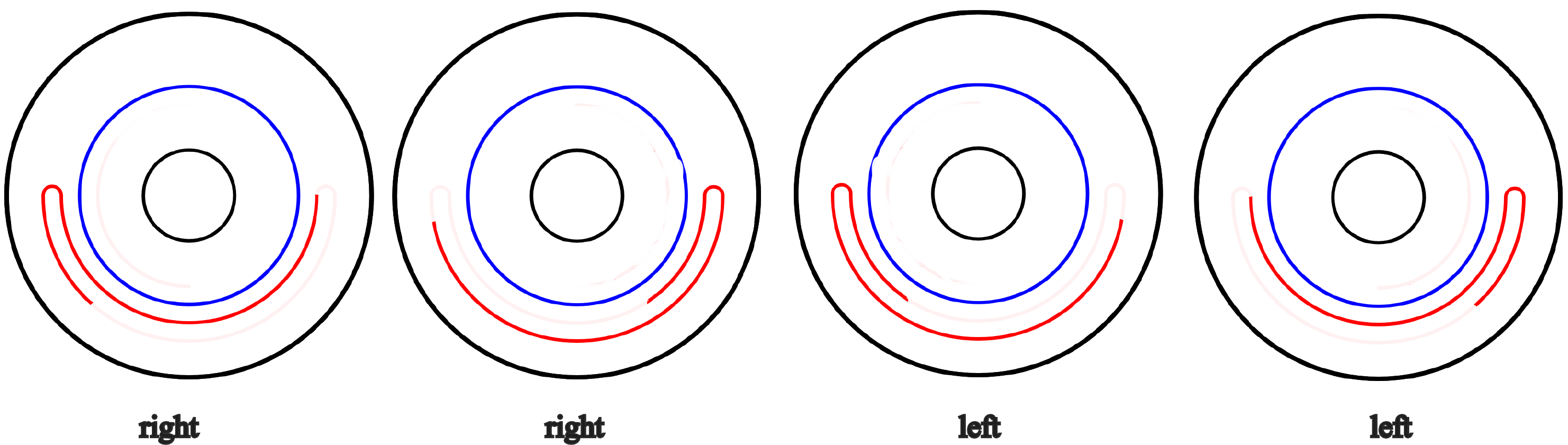}
		\caption {Essential vertices for bordered quasisection, types $IV_{tf}$, right and left. }
		\label{fig:3nonsymm}
	\end{figure}

	Let $x$ be an essential vertex of type $I_{bb}$ (of type of type $I_{bf}$, respectfully). We order the boundary curves (the  pair "boundary curve, fold curve", respectfully) in the same way as we ordered folds for essential vertices of
type $I$. 

	Namely, imagine a point $y$ that goes around $x$  in the ccw direction, starting from a place with 
	 the minimal number of points in $\pi^{-1}(y)\cap\mathcal{ Q}$, and
 apply the rule ''the point $y$ meets the first  curve   first''. 
	
{ Set
		$n(x)$  be the number of simple circles of  $\pi^{-1}(y)\cap \mathcal{Q}$  lying between the first and the second  curves, 
		if one counts from the \textbf{first}  curve in the direction of the fiber.}
	
	Set also $k(x)$ be the number of simple circles  lying between the  curves, 
	if one counts from the \textbf{second}  curve in the direction of the fiber.  
	
\medskip

\medskip
	
	It is clear that:
	\begin{lemma}\begin{enumerate}
	               \item For essential vertices of types $I_{bb}$ and $I_{bf}$,\begin{enumerate}
			\item The portrait 
			defines (and  is defined by)
			two
			numbers, $n$ and $k$, and the type (right or left)  for  $I_{bf}$.
			\item For the  line symmetry image, $n$ and $k$ interchange.
			\item  The number $n+k$ cannot be zero.
		\end{enumerate}
	              
	               \item 	For essential vertices of type $III_{bs}$, the  portrait  defines
		(and  is defined by)  its type (right or left) and the  number $r$ of simple circles ($r$ might be zero).     
	             \end{enumerate}
\end{lemma}

\medskip

Now we assign\textit{ weights }to new essential singular vertices.  

	\begin{definition}\begin{enumerate}
			\item  For an essential vertex  $x$ of type $I_{bb}$, set $$\mathcal{W}_{bb}(x)=\mathcal{W}_{bb}(n,k)=\frac{(n-k)}{(n+k)(n+k+1)(n+k+2)}.$$
			
			\item For a right essential vertex $x$ of type $I_{bf}$  set $$\mathcal{W}_{bf}(x)=\mathcal{W}_{bf}(n,k, R)=\frac{(n-k)}{(n+k)(n+k+3)}\Big(\frac{1}{n+k+1}+\frac{1}{n+k+2}\Big). $$
			
			For a left essential vertex $x$ of type $I_{bf}$ set $$\mathcal{W}_{bf}(x)=\mathcal{W}_{bf}(n,k, L)=\frac{(k-n)}{(n+k)(n+k+3)}\Big(\frac{1}{n+k+1}+\frac{1}{n+k+2}\Big). $$
			
			\item For a right essential vertex $x$ of type $III_{bs}$, set $$\mathcal{W}_{bs}(x)=\mathcal{W}_{bs}(r, R)= \frac{1}{(r+1)(r+2)}.$$
			
			For a left essential vertex $x$ of type $III_{bs}$, set $$\mathcal{W}_{bs}(x)=\mathcal{W}_{bs}(r, L)= \frac{-1}{(r+1)(r+2)}.$$

\item For a right essential vertex $x$ of type $IV_{tf}$, set $$\mathcal{W}_{tf}(x)=\mathcal{W}_{tf}(r, R)= \frac{1}{2(r+1)(r+2)}.$$
			
			For a left essential vertex $x$ of type  $IV_{tf}$, set $$\mathcal{W}_{tf}(x)=\mathcal{W}_{tf}(r, L)= \frac{-1}{2(r+1)(r+2)}.$$
		\end{enumerate}
		
	\end{definition}

	\begin{theorem}\label{ThmMain2}  
		The Euler number of a circle bundle with a generic (possibly bordered) quasisection $\mathcal{Q}$ equals the sum of weights of essential singular vertices:
		
		$$\mathcal{E}= \sum_I \mathcal{W}_{ff}(x_i)+ \sum_{II} \mathcal{W}_{p}(x_i)+\sum_{III} \mathcal{W}_{fs}(x_i)+$$
$$\sum_{I_{bb}} \mathcal{W}_{bb}(x_i)+ \sum_{I_{bf}} \mathcal{W}_{bf}(x_i)+\sum_{III_{bs}} \mathcal{W}_{bs}(x_i)+\sum_{IV_{tf}} \mathcal{W}_{tf}(x_i).$$
		
		This formula is unique, that is, no other assignment of weights gives a valid expression for the Euler number.
	\end{theorem}

\begin{proof} (a)  We apply direct computations of the averaged index of sections, exactly as was done in Theorem \ref{ThmMain}. 

(b) Uniqueness follows analogously to Section \ref{SecUni}. That is,
\begin{enumerate}
  \item Uniqueness of $W_{bb}$ is proven analogously to uniqueness of $W_{ff}$.
  \item Next, uniqueness of $W_{tf}$  comes from uniqueness of $W_{bb}$ and Fig. \ref{FigTermFold}.
  \item Uniqueness of $W_{bs}$  comes from Fig. \ref{FigCrossingDisks}, left.
\item Finally, $W_{bf}$ comes from the following surgery. Whenever we have a border line and a fold whose projections intersect, let us cut off a neighborhood of a point $y$ lying on the fold which projects to the projection of the border line, see Fig. \ref{FigCrossingDisks}, right. We get a new quasisection with eliminated  $I_{bf}$ singularity. So the weight $W_{bf}$ expresses in terms of already known weights.
\end{enumerate}
\end{proof}

	\begin{figure}[h]
		\includegraphics[width=0.6\linewidth]{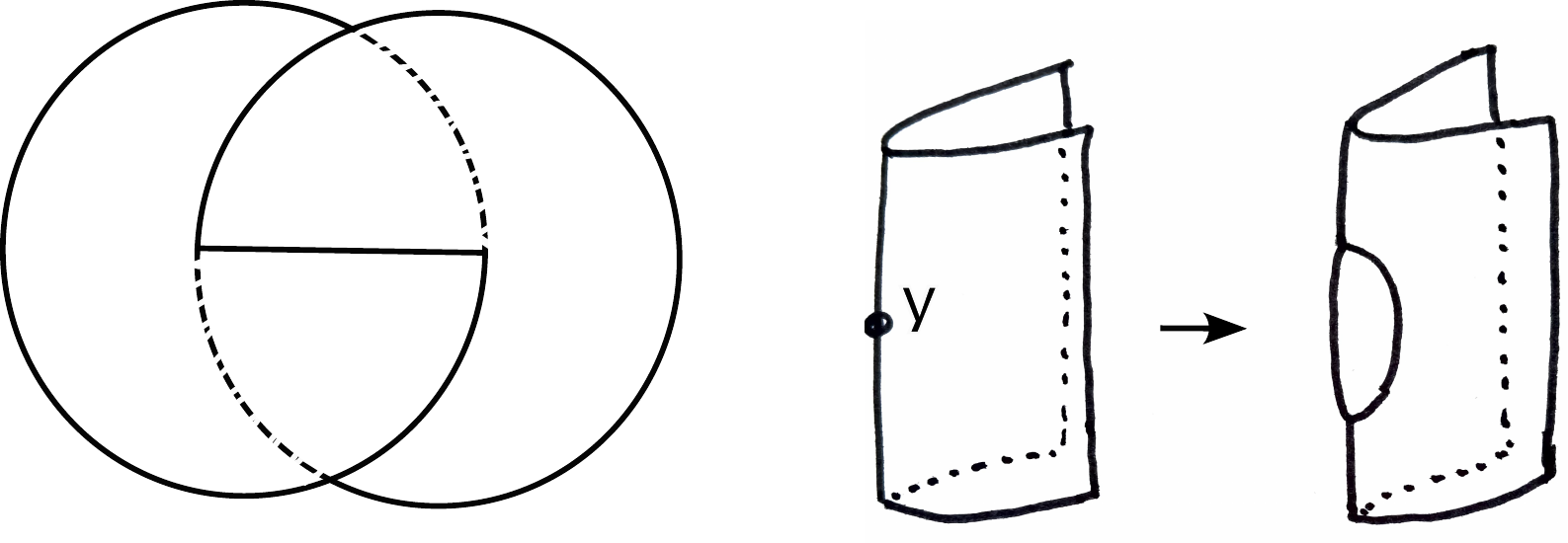}
		\caption {Left: adding two crossing disks amounts to adding two singular vertices of type $III_{bs}$ and two more of type $I_{bb}$.
Right: this surgery reduces $I_{bf}$ singularity to other types.}
		\label{FigCrossingDisks}
	\end{figure}
	
	\bigskip

\textbf{Acknowledgement. } 
  This work is supported by the Russian Science Foundation (project 25-11-00058)

\end{document}